\documentclass[lettersize,journal]{IEEEtran}
\usepackage{amsmath, amssymb, amsfonts, amsbsy, amsthm, latexsym, epsfig}
\usepackage{algorithmic}
\usepackage{algorithm}
\usepackage{array}
\usepackage[caption=false,font=normalsize,labelfont=sf,textfont=sf]{subfig}
\usepackage{textcomp}
\usepackage{stfloats}
\usepackage{url}
\usepackage{verbatim}
\usepackage{graphicx}
\usepackage{cite}
\usepackage{color}
\usepackage{mathtools}
\usepackage[hidelinks]{hyperref} 
\mathtoolsset{showonlyrefs}
\newtheorem{theo}{Theorem}[section]
\newtheorem{lemma}[theo]{Lemma}
\newtheorem{prop}[theo]{Proposition}
\newtheorem{coro}[theo]{Corollary}

\newtheorem{rema}[theo]{Remark}

\def\eps{\varepsilon}

\def\w{\omega}

\def\N{\mathbb{N}}
\def\Z{\mathbb{Z}}
\def\R{\mathbb{R}}

\def\C{\mathbb{C}}

\def\cH{\mathcal{H}}

\def\supp{\text{Supp}}

\def\namedlabel#1#2{\begingroup
	#2
	\def\@currentlabel{#2}
	\phantomsection\label{#1}\endgroup}

\begin{document}

\title{Random periodic sampling patterns for shift-invariant spaces}

\author{Jorge Antezana, Diana Carbajal and Jos\'e Luis Romero
	\thanks{© 2023 IEEE.  Personal use of this material is permitted.  Permission
		from IEEE must be obtained for all other uses, in any current or
		future media, including reprinting/republishing this material for
		advertising or promotional purposes, creating new collective works,
		for resale or redistribution to servers or lists, or reuse of any
		copyrighted component of this work in other works. DOI: 10.1109/TIT.2023.3320289.}}

\maketitle

	\begin{abstract}
		We consider multi-variate signals spanned by the integer shifts of a set of generating functions with distinct frequency profiles and the problem of reconstructing them from samples taken on a random periodic set. We show that such a sampling strategy succeeds with high probability provided that the density of the sampling pattern exceeds the number of frequency profiles by a logarithmic factor.
		The signal model includes bandlimited functions with multi-band spectra. While in this well-studied setting delicate constructions provide sampling strategies that meet the information theoretic benchmark of Shannon and Landau, the sampling pattern that we consider provides, at the price of a logarithmic oversampling factor, a simple alternative that is accompanied by favorable a priori stability margins (snug frames). More generally, we also treat bandlimited functions with arbitrary compact spectra, and different measures of their complexity and approximation rates by integer tiles. 
		At the technical level, we elaborate on recent work on relevant sampling, with the key difference that the reconstruction guarantees that we provide hold uniformly for all signals, rather than for a subset of well-concentrated ones. This is achieved by methods of concentration of measure formulated on the Zak domain.
	\end{abstract}


	\section{Introduction}
\IEEEPARstart{T}{he} classical Nyquist–Shannon sampling theorem provides a link between analog signals and their values at discrete spatial locations. The result concerns \emph{bandlimited signals}, modeled by the \emph{Paley-Wiener space}
\begin{align*}
	PW_{[-1/2,1/2]^d} = \{f\in L^2(\R^d)\,:\, \supp(\hat{f})\subseteq [-1/2,1/2]^d\},
\end{align*}
where $\hat{f}(\omega)= \int_{\R^d} f(x) e^{-2\pi i \langle x, \omega\rangle} dx$ denotes the Fourier transform.
Heuristically, bandlimited functions have one degree of freedom per unit of space, because of the restriction on their oscillation imposed by their Fourier support. The Nyquist–Shannon theorem formalizes this intuition, showing that the energy of a bandlimited function is exactly retained by its samples on the integer grid:
\begin{align}\label{eq_i1}
	\|f\|_2^2 = \sum_{k \in \Z^d} |f(k)|^2, \qquad f \in PW_{[-1/2,1/2]^d}.
\end{align}
While more sophisticated sampling theorems also cover possibly non-uniform sampling patterns, the Nyquist–Shannon excels by its simplicity and is a corner stone of signal processing \cite{843002}.

Much research has been motivated by the quest for comparable results for general Paley-Wiener spaces
\begin{align*}
	PW_\Omega = \{f\in L^2(\R^d)\,:\, \supp(\hat{f})\subseteq\Omega\},
\end{align*}
defined with respect to a compact set $\Omega \subset \R^d$ \cite{OU16}. For example, sampling problems related to \emph{multi-tile spectra} of the form
\begin{align}\label{eq_i2}
	\Omega=\bigcup_{i=1}^k [-1/2,1/2]^d+\ell_i, 
\end{align}
with $\ell_1,\dots,\ell_k \in \Z^d$ distinct, are relevant in wireless communications as models for physical or commercial restrictions on certain frequency intervals. Denoting the \emph{spectral diameter} by $N:=2\max\{|\ell_1|_\infty, \ldots, |\ell_k|_\infty\}$, a rescaling of \eqref{eq_i1} shows that the grid $\tfrac{1}{N+1} \Z^d$ captures the energy of each function in $PW_\Omega$. However, in many situations of interest, the spectrum \eqref{eq_i2} contains a relatively small number $k$ of active frequency bands, and one could expect $k$ rather than $N$ to govern the required sampling rate. As we discuss below, this intuition is indeed correct, even though the sampling grids need to be replaced by slightly more complex patterns.

Whenever a set $\Lambda \subset \R^d$ provides the \emph{sampling estimates}
\begin{align}\label{eq_i4}
	A \|f\|_2^2 \leq \sum_{\lambda \in \Lambda} |f(\lambda)|^2 \leq B \|f\|_2^2, \qquad f \in PW_\Omega,
\end{align}
for some constants $A,B>0$, Landau's celebrated density theorem \cite{landau1967sampling,La67}
provides the following comparison between the Lebesgue measure of the spectrum $\Omega$ and the so-called \emph{lower Beurling density} of the set $\Lambda$:
\begin{equation}\label{eq_i3}
	D^{-}(\Lambda) := \liminf_{R \to \infty} \inf_{x \in \R^d} \frac{\# \Lambda \cap ([-R/2,R/2]^d + x)}{R^d} \geq |\Omega|.
\end{equation}
The density of a sampling lattice is $D^{-}(\tfrac{1}{L} \Z^d) = L^d$ whereas the measure of the multi-tile spectrum in \eqref{eq_i2} is $|\Omega|=k$.

While in general it is difficult to achieve the benchmark on the right-hand side of \eqref{eq_i3} with uniform sampling patterns, certain slightly more complex patterns accomplish the task. Consider the (possibly) \emph{non-uniform periodic set}
\begin{align}\label{eq_i5}
	\Lambda = \{x_1, \ldots, x_m\} + \Z^d,
\end{align}
and a multi-tile spectrum $\Omega$ as in \eqref{eq_i2}. As shown in \cite{LS97,Far90,Mar06, Ko15}, for a generic choice of the set $\{x_1, \ldots, x_m\} \subset [-1/2,1/2]^d$ with $m=k$, $\Lambda$ satisfies the sampling estimates \eqref{eq_i4}. As ${D^{-}(\Lambda)=m=k}$, such sets saturate the density benchmark \eqref{eq_i3}. Other sampling patterns, such as quasi-crystals, provide similar or comparable results for arbitrary spectra \cite{MM08, OU08, GL14, KN15, DL22}.

Though remarkable, the mentioned results on generic periodic sampling sets \eqref{eq_i5} have certain disadvantages that limit their use in practice. First, the practitioner would need to carefully choose $\{x_1, \ldots, x_m\}$ so as to avoid a certain algebraic variety of exceptional sets. Second, each such choice leads to potentially different sampling constants in \eqref{eq_i4}, and no simple a priori bounds are provided. The purpose of this note is to provide a variant of the mentioned results on non-uniform periodic sampling for multi-tile spectra, where the set $\{x_1, \ldots, x_m\}$ is chosen uniformly at random and $m \approx k \log(k)$. 
With respect to deterministic results, our work represents the following trade-off. On the one hand, the required sampling density, $k \log(k)$, exceeds the benchmark value $k$, albeit by a logarithmic factor. On the other hand, the sampling strategy is simple and is accompanied with
sampling estimates \eqref{eq_i4} with explicit and possibly very favorable stability margins. Second, we formulate our result more generally in the setting of \emph{shift-invariant spaces} \cite{MR1882684,843002}. These spaces are spanned by the integer translates of certain generating functions, which are assumed to have distinct but possibly infinitely supported frequency profiles.

\section{Results and context}
\subsection{Results}
Our main result uses the \emph{Zak transform} of a function ${f \in L^2(\R^d)}$,
\begin{equation}
	Z_{f}(x,\w) = \sum_{\ell\in\Z^d} f(x+\ell) e^{-2\pi i \langle \ell,\w\rangle},
	\qquad 	(x,\w)\in\R^{2d},
\end{equation}
and reads as follows.

\begin{theo}\label{thm:main-thm}
	Let  $\Phi=\{\phi_1,\dots,\phi_k\} \subset L^2(\R^d)$ be a set of functions whose integer translates ${E(\Phi)= \left\{\phi_i(\cdot- \ell)\,:\,i=1,\dots,k,\,\ell\in\Z^d \right\}}$ are orthonormal, and consider the \emph{shift-invariant space}
	\begin{equation}\label{eq:SIS}
		V=S(\Phi):=\overline{\text{span}}\left\{\phi_i(\cdot- \ell)\,:\, i=1,\dots,k,\,\ell\in\Z^d \right\}.
	\end{equation} 
	Assume the following:
	\begin{enumerate}
		\item\label{item:cond-0}({Distinct frequency profiles}).
		There are functions $\{\psi_1, \ldots, \psi_k\} \subset L^2(\R^d)$
		and base frequencies $\ell_1,\dots,\ell_k\in\Z^d$ such that
		\begin{equation}\label{eq:phi_i}
			\phi_i = e^{2\pi i \langle \ell_i, \cdot\rangle } \psi_i,\quad i=1,\dots,k.
		\end{equation}
		\item\label{item:cond-1}({Bounded quadratic periodization in time})
		$$\sup_{i=1,\ldots,k} \sup_{x\in[-1/2,1/2]^d} \sum_{\ell\in\Z^d} |\psi_i(x+\ell)|^2 < \infty.$$
		\item\label{item:cond-2}({Bounded periodization in frequency}) 
		$$C:=\sup_{i=1,\ldots,k}\underset{\w\in [-1/2,1/2]^d}{\text{\rm ess sup }}\sum_{\ell\in\Z^d} |\hat{\psi}_i(\w+\ell)| < \infty.$$
		\item\label{item:cond-3}(Smooth Zak transform) 
		\begin{align}
		&K:=\underset{\substack{\w,\w'\in [-1/2,1/2]^d\\ \w\neq\w'}}{\text{\rm ess sup }} \sup_{\substack{x\in[-1/2,1/2]^d\\i=1,\ldots,k}}\\
		&\qquad\qquad\qquad  \frac{|Z_{\hat\psi_i}(\w,x)-Z_{\hat\psi_i}(\w',x)|}{\|\w-\w'\|_{\infty}} < \infty.
		\end{align}
	\end{enumerate}		
	Then $V$ is a reproducing kernel Hilbert space. Moreover, let $X=\{x_1,\dots,x_m\}$ be a set of independent random variables that are uniformly distributed in  $[-1/2,1/2]^d$, let $\varepsilon>0$ and $0<\alpha<1$.  If
	\begin{equation}
		m\geq10 \cdot \frac{C^2}{\alpha^2}\cdot k\cdot  \log\left( \frac{2k}{\varepsilon} \left( \frac{2k C K}{\alpha}  +1\right)^d\right),
	\end{equation}
	then the sampling inequalities
	\begin{equation}\label{eq:sampling-ineq-m}
		m(1-\alpha) \|f\|^2 \leq  \sum_{r=1}^m  \sum_{\ell\in\Z^d} |f(x_r+\ell)|^2 \leq m(1+\alpha) \|f\|^2,
	\end{equation}
	hold for every $f\in V$ with probability at least $1-\varepsilon$.
\end{theo}
Some remarks are in order. First, under the assumptions of Theorem \ref{thm:main-thm}, $Z_{\hat\psi_i}(\w,x)$ is well-defined for every $x$ and almost every $\w$, see Remark \ref{rem:Zak}. The claim that $V$ is a reproducing kernel Hilbert space means that each function in $V$ is continuous, and the evaluation function $V \ni f \mapsto f(x)$ is bounded for each $x \in \R^d$.

Second, under a moderate amount of \emph{oversampling}, the ratio between the upper and lower stability bounds in \eqref{eq:sampling-ineq-m} is close to $1$, much in analogy to the classical Shannon-Nyquist theorem. This property, called \emph{snugness} in the engineering literature, means that the recovery of a function from its samples is numerically efficient. The merit of Theorem \ref{thm:main-thm} is in achieving this while keeping the sampling density close to the benchmark value given by Landau's theorem \cite{landau1967sampling,La67}. Note that, although a simple  \emph{oscillation estimate} provides a snug sampling set, it would need samples at the potentially much larger rate dictated by the maximal oscillation	$\max_j |\ell_j|_\infty$.

We note that in Theorem \ref{thm:main-thm} the class $\{\psi_1,\dots,\psi_k\}$ admits repetition. In fact, to cover the model case of Paley-Wiener spaces with multi-tile spectra, we can let $\hat{\psi}_i$ be the indicator function $\hat{\psi}_i=\chi_{[-1/2,1/2]^d}$. Then $Z_{\hat\psi_i}(\w,x)=1$, $K=0$ and $C=1$, and we have the following corollary of Theorem \ref{thm:main-thm}.

\begin{coro}\label{coro:union-cubes}
	Let $\Omega=\bigcup_{i=1}^k [-1/2,1/2]^d+\ell_i$, with $\ell_1,\dots,\ell_k$ distinct points in $\Z^d$. Let $X=\{x_1,\dots,x_m\}$ be independent random variables uniformly distributed in  $[-1/2,1/2]^d$, $\varepsilon>0$ and $0<\alpha<1$. If
	\begin{equation}
		m\geq \frac{10}{\alpha^2}\cdot k\cdot \log\left( \frac{2k}{\varepsilon}\right),
	\end{equation}
	then the sampling inequalities
	\begin{equation}\label{eq:sampling-ineq-union-cubes}
		m(1-\alpha) \|f\|^2 \leq  \sum_{r=1}^m  \sum_{\ell\in\Z^d} |f(x_r+\ell)|^2 \leq m(1+\alpha) \|f\|^2,
	\end{equation}
	 hold for every $f\in PW_\Omega$ with probability at least $1-\varepsilon$.
\end{coro}
We remark that the bounds in Corollary \ref{coro:union-cubes} are independent of the ambient dimension $d$. We prove Corollary \ref{coro:union-cubes} in Section \ref{sec:thm3}.

Finally, we discuss Paley-Wiener spaces with general spectra. We define the \emph{tiling complexity index} of a compact set ${\Omega \subset \R^d}$ as the minimum number $N=N_\Omega \in \N$ for which there exist disjoint measurable sets ${Q_1, \ldots, Q_N\subseteq [-1/2,1/2]^d}$ and $B_1, \ldots, B_N \subseteq \Z^d$ such that
\begin{align}\label{eq_compi}
	\Omega = \bigcup_{n=1}^N Q_n + B_n,
\end{align}
up to null measure sets. If $\Omega$ is as in Corollary \ref{coro:union-cubes} then ${N_\Omega=1}$. On the other hand, for
more complex sets, $N_\Omega$ may be large; a crude upper bound is $N_\Omega \leq \binom{\# L_\Omega}{k_\Omega}$, where 
\begin{align}\label{eq:k-level}
	&k_\Omega:= \underset{\w\in [-1/2,1/2]^d}{\text{\rm ess sup }} \sum_{\ell \in \Z^d}\chi_{\Omega}(\w+\ell),
	\\
	\label{eq:L_Omega}
	&L_\Omega:= \left\{\ell \in \Z^d\,:\,\left|\Omega\cap \left([-1/2,1/2)^d+\ell\right)\right| >0\right\}.
\end{align}

The following result generalizes Corollary \ref{coro:union-cubes} to general compact spectra.

\begin{theo}\label{coro:k-tiles}
	Let $\Omega\subset \R^d$ be a compact set with tiling complexity index $N_\Omega$ and let $k=k_\Omega$ be given by \eqref{eq:k-level}. Let $X=\{x_1,\dots,x_m\}$ be independent random variables uniformly distributed in  $[-1/2,1/2]^d$, $\varepsilon>0$, $0<\alpha<1$, and
	\begin{equation}
		m\geq \frac{10}{\alpha^2}\cdot k\cdot \log\left( \frac{2N_\Omega k}{\varepsilon}\right).
	\end{equation}
	Then the sampling inequalities in \eqref{eq:sampling-ineq-union-cubes} hold with probability at least $1-\varepsilon$.
\end{theo}

\medskip

Finally we consider a second variant of Corollary \ref{coro:union-cubes} where complexity is measured in terms of boundary regularity. We say that a set $\Omega$ has \textit{maximally Ahlfors regular boundary} if there exists a constant $\kappa_{\partial\Omega}$ such that
\begin{align*}
\cH^{d-1}(\partial\Omega \cap B_r(x))\geq \kappa_{\partial \Omega} r^{d-1},\quad \forall\, 0<r\leq \cH^{d-1}(\partial\Omega)^{\frac{1}{d-1}},\\
\forall x\in\partial\Omega.
\end{align*}
Here $\cH^{d-1}$ denotes the $(d-1)$-Hausdorff measure. Most practically relevant domains enjoy this condition.

\begin{theo}\label{teo general set}
	Let $\Omega\subset \R^d$ be a compact set with maximally Ahlfors
	regular boundary with constant $\kappa_{\partial \Omega}$. Given $\varepsilon>0$ and $0<\alpha<1$, let
	$$
	\rho = \frac{\kappa_{\partial \Omega}}{C_d} \cdot \frac{|\Omega|}{ \cH^{d-1}(\partial \Omega)},
	\quad \mbox{and}\quad 
	m\geq \frac{20}{\alpha^2} \cdot \frac{|\Omega|}{\rho^d} \cdot \log\left(\frac{4}{\varepsilon} \cdot \frac{|\Omega|}{\rho^d}\right),
	$$
	where  $C_d$ denotes an adequate constant that depends only on $d$. If $x_1,\dots,x_m$ are independent random variables uniformly distributed in 
	$\left[-(2\rho)^{-1},(2\rho)^{-1}\right]^d$, 
	then the sampling inequalities
	\begin{equation}\label{eq:sampling-ineq-general}
		m(1-\alpha)\,\rho^d \,\|f\|^2 \leq  \sum_{r=1}^m  \sum_{\ell\in\frac{1}{\rho}\Z^d} |f(x_r+\ell)|^2 \leq m(1+\alpha)\,\rho^d\, \|f\|^2
	\end{equation}
	hold for every $ f\in PW_{\Omega}$ with probability at least $1-\varepsilon$.
\end{theo}

We prove Theorem \ref{coro:k-tiles} and Theorem \ref{teo general set} also in Section \ref{sec:thm3}.

\medskip

\subsection{Related literature}
Non-uniform periodic sampling goes back at least to Papoulis \cite{Pa77} and its characterization in terms of the Zak transform is part of the folklore of the theory of shift-invariant spaces \cite{HL06,Bow01}. In particular, the usefulness of periodic sampling for multi-band spectra is well known \cite{843002}. Our work is motivated by the investigation of (slight) oversampling in the periodic setting and in particular by the question of whether the use of random sampling patterns can save practitioners from checking delicate conditions such as the invertibility of Zak matrices (that is, matrices whose entries are Zak transforms) and the uniformity of such statements on underlying parameters. The point pattern that we study is very natural as a model for many common sampling practices. For example, it has been considered under the name of \emph{random interleaved sampling} in \cite{PTWY12} to study the performance of analog-to-digital conversion of functions bandlimited to the unit interval.	

The literature on random sampling is extensive. For example, \cite{SU96,CL97} investigate sampling in classical Paley-Wiener spaces by checking that deterministic conditions for sampling apply to certain random perturbations of the integer grid, while \cite{BG05,KR08} investigates random sampling of frequency-sparse multivariate trigonometric polynomials. The notion of \emph{relevant sampling} was introduced in \cite{BG10} to study Paley-Wiener spaces whose spectrum is a high-dimensional cube. In this setting, functions are sampled on a finite random set and the reconstruction problem is restricted to only those functions concentrated near the sampling domain, which facilitates the use of random matrix theory. This method has been streamlined in \cite{BG13} and extended to several other signal models including shift-invariant spaces \cite{FX19,YW13,GPS23,JZ23,JZ20} and more general reproducing kernel Hilbert spaces \cite{MR4241985}. In contrast, our work provides sampling inequalities valid in the full (infinite dimensional) Paley-Wiener space, and are efficient even if the spectrum has multiple connected components. At the technical level we also draw on the argument of \cite{BG13}. Crucially, we do not argue directly on the signal domain, but rather on the (compact) Zak domain, which also allows us to transparently keep track of multi-component frequency profiles, c.f. \eqref{eq:phi-to-psi}.

Whereas the sampling set produced in Theorem \ref{thm:main-thm} is considered of high-quality because it simultaneously achieves stability constants close to each other (snugness) and a sampling density close to Landau's benchmark \eqref{eq_i3} (low redundancy), more classical work prioritizes the latter quality criterion. The minimal redundancy is achieved by \emph{complete interpolating} systems, and non-uniform periodic constructions go back to at least \cite{LS97}. Borrowing inspiration from \cite{Far90}, such constructions were later extended to $d>1$ in \cite{Mar06}, to cover Paley-Wiener spaces whose spectrum is a finite union of equally sized cubes, and refined in \cite{Ko15} to cover arbitrary multi-tiles. On the Fourier domain, complete interpolating sequences for $PW_\Omega$ are the frequency sets of Riesz bases of complex exponentials in $L^2(\Omega)$. In this formulation, the problem of existence of bases or low redundancy systems of exponentials has attracted significant attention; see, e.g.,  \cite{OU08,MM08, GL14,KN15,DL22,AAC15,CC18,CHM21}.

Finally, we mention the remarkable article \cite{MR3415581}, where, as a technical step towards the main result, it is proved that the Paley-Wiener space with spectrum as in \eqref{eq_i2} admits a sampling set with both frame bounds $\approx k$ and Beurling density $\approx k$ (whereas Theorem \ref{thm:main-thm} only achieves $\approx k \log(k)$). While the existence of the sampling set is derived from the Marcus, Spielman and Srivastava solution to the paving conjecture \cite{MSS15}, and therefore lacks the simplicity of the random pattern in Theorem \ref{thm:main-thm}, the investigation of algorithmic variants of this approach is an active field of research \cite{BAU23}.

\section{Proof of the main theorem}\label{sec:thm1}

We will prove Theorem \ref{thm:main-thm} by following a series of steps.
Throughout this section, $\Phi = \{\phi_1,\dots,\phi_k\}\subset L^2(\R^d)$ is a set of functions such that $E(\Phi)$ is an orthonormal system, and $\psi_1,\dots,\psi_k$ are the functions introduced in Condition \ref{item:cond-0}, which are assumed to satisfy Conditions \ref{item:cond-1}-\ref{item:cond-3}. We use the notation \eqref{eq:SIS} for the shift-invariant spaces spanned by $\Phi$.

\begin{rema}\label{rem:Zak}
	If $f\in L^2(\R^d)$, then $\{f(x+\ell)\}_{\ell\in\Z^d}\in\ell^2(\Z^d)$ for a.e. $x\in\R^d$ and thus $Z_f(x,\w)$ is defined almost everywhere on $\R^{2d}$.
	Whenever $\{f(x+\ell)\}_{\ell\in\Z^d} \in \ell^2(\Z^d)$ \textit{for every} $x\in\R^d$, the Zak transform $Z_f(x,\w)$ is defined for every $x\in\R^d$ and for a.e. $\w\in\R^d$. If we also have that
	$$
	\sum_{\ell\in\Z^d} \hat{f}(\cdot+\ell)\in L^2([-1/2,1/2]^d),
	$$ 
	then 
	\begin{equation}\label{eq:Zak-Fourier}
		Z_f(x,\w) = e^{2\pi i \langle x, \w\rangle} Z_{\hat{f}}(\w,-x),\quad \forall x\in\R^d,\, \text{for a.e. }\w\in\R^d.
	\end{equation}
\end{rema}

In Condition \ref{item:cond-3} we implicitly assume that the Zak transforms are defined on every $x\in [-1/2,1/2]^d$. This is justified by Conditions \ref{item:cond-1} and \ref{item:cond-2}. Indeed, these conditions assert that $\{\psi_i(x+\ell)\}_{\ell\in\Z^d}\in \ell^2(\Z^d)$ for every $x\in\R^d$ and ${\sum_{\ell\in\Z^d} \hat{\psi}_i(\cdot+\ell)\in L^2([-1/2,1/2]^d)}$. Thus, $Z_{\psi_i}(x,\w)$ is well defined \textit{for every} $x\in\R^d$ and for a.e. $\w\in\R^d$, and \eqref{eq:Zak-Fourier} holds for $f=\psi_i$, for $i=1,\dots,k$, for every $x\in\R^d$ and for a.e. $\w\in\R^d$. In light of Condition \ref{item:cond-0}, we conclude that for every $x\in\R^d$ and for a.e. $\w\in\R^d$,
\begin{equation}\label{eq:phi-to-psi}
	Z_{\hat{\phi}_i} (\w,x) = e^{2\pi i \langle x,\ell_i\rangle}  Z_{\hat\psi_i}(\w,x).
\end{equation}

Furthermore, Condition \ref{item:cond-1} says that all $\phi_i$ have bounded quadratic periodization in time. As a consequence, $V=S(\Phi)$ is a \emph{reproducing kernel Hilbert space}. Indeed, for $f\in V$ and $x\in\R^d$,
\begin{align*}
	|f(x)|&\leq \sum_{\ell\in\Z^d}\sum_{i=1}^{k} |\langle f,\phi_i(\cdot-\ell)\rangle | |\phi_i(x-\ell)| \\ 
	&\leq (kD)^{\frac{1}{2}} \bigg(\sum_{\ell\in\Z^d}\sum_{i=1}^{k} |\langle f,\phi_i(\cdot-\ell)\rangle |^2 \bigg)^{\frac{1}{2}}  
	\\&= (kD)^{\frac{1}{2}} \|f\|,
\end{align*}
where $D:=\sup_{x\in[-1/2,1/2]^d} \sum_{\ell\in\Z^d} |\phi_i(x-\ell)|^2$.

\medskip

For a set of points $X= \{x_1,\dots, x_m\}\subset [-1/2,1/2]^d$ and for a.e. $\w\in [-1/2,1/2]^d$, we define the Zak matrix $Z_\Phi(X,\w)\in\C^{k\times m}$ as
\begin{equation}\label{eq:matrix-Zac}
	\left( Z_\Phi (X,\w) \right)_{i,j} = Z_{\phi_i}(x_j,\w),\quad i \in \{1,\dots, k\}, j\in \{1,\dots,m\}.
\end{equation}

The following proposition connects the sampling property of $X+\Z^d$ in $V=S(\Phi)$ with the associated Zak matrix.
\begin{prop}\label{propo:sampling-condition}
	The sampling inequalities
	\begin{equation}\label{eq:sampling-inequalities-AB}
		A\|f\|^2 \leq \sum_{r=1}^m \sum_{\ell\in\Z^d} |f(x_r+\ell)|^2 \leq B\|f\|^2,\quad \forall \,f\in V,
	\end{equation}
	hold for certain constants $A,B>0$ if and only if for a.e. $\w\in [-1/2,1/2]^d$,
	\begin{equation}\label{eq:sampling-ineq-Zak}
		A\|a\|^2\leq \left\langle \overline{Z_\Phi(X,\w)} Z_\Phi(X,\w)^t a,a\right\rangle \leq B\|a\|^2,\quad \forall a\in\C^k.
	\end{equation}
\end{prop}
The proposition is folklore in the theory of sampling in shift-invariant spaces (see e.g. \cite{Rom06}), but we provide a brief proof for completeness. The proof only uses the fact that $V$ is a reproducing kernel Hilbert space and $E(\Phi)$ is an orthonormal basis of $V$.
\begin{proof}
	For a function $f\in L^2(\R^d)$, its \emph{fiber} at $\w$,
	\begin{equation}\label{eq:fiber}
		\mathcal{T}f(\w):=\{\hat{f}(\w+\ell)\}_{\ell\in\Z^d},\quad\w\in[-1/2,1/2]^d,
	\end{equation} belongs to $\ell^2(\Z^d)$ for a.e. $\w\in [-1/2,1/2]^d$. 	
	
	Denote by $K_x\in V$  the reproducing kernel at the point $x\in\R^d$. Since $K_x(\cdot - \ell) = K_{x+\ell}$ for every $\ell\in\Z^d$, the sampling inequalities \eqref{eq:sampling-inequalities-AB} mean that  ${E(K):=\{K_{x_r}(\cdot-\ell)\,:\, r=1,\dots,m, \,\ell\in\Z^d\}}$ is a frame of $V$ with frame bounds $A,B$.
	
	The reproducing kernels $K_{x_r}$ can be expanded in
	the orthonormal basis $E(\Phi)$ 
	as
	\begin{equation}\label{eq:expression-kernel}
		K_{x_r}(y) = \sum_{\ell\in\Z^d} \sum_{i=1}^k \overline{\phi_i(x_r+\ell)} \phi_i(y+\ell), \quad y\in\R^d.
	\end{equation}
	With the notation
	\begin{equation}
	\hat{K}(\w) := \left(\mathcal T K_{x_1}(\w),\dots,\mathcal T K_{x_m}(\w)\right)\in\C^{\Z^d\times \{1,\dots,m\}}, 
	\end{equation}
	and \begin{equation}
		\hat{\Phi}(\w) := \left(\mathcal T \phi_{1}(\w),\dots,\mathcal T \phi_{k}(\w)\right)\in\C^{\Z^d\times \{1,\dots,k\}},
	\end{equation}
	$\w\in [-1/2,1/2]^d,$
	the expansion \eqref{eq:expression-kernel} reads
	$$\hat{K}(\w) = \hat{\Phi}(\w) \overline{Z_\Phi(X,\w)},
	$$
	where $\hat{K}(\w)\in \ell^2(\Z^d)^m$ and $\hat{\Phi}(\w) \in \ell^2(\Z^d)^k$ for a.e. $\w\in[-1/2,1/2]^d$.
	The \textit{dual gramian} associated with $E(K)$ is defined as the infinite matrix $\widetilde{G}_K(\w) := \hat{K}(\w) \hat{K}(\w)^*\in \C^{\Z^d\times \Z^d}$, and thus admits the factorization 
	$$\widetilde{G}_K(\w) = \hat{\Phi}(\w) \overline{Z_\Phi(X,\w)} Z_\Phi(X,\w)^t  \hat{\Phi}(\w)^*.$$
	By the theory of shift-invariant spaces (see e.g. \cite{RS95} or \cite[Theorem 2.5]{Bow01}), $E(K)$ is a frame of $V$ with bounds $A,B$ if and only if 
	\begin{align}\label{eq:bounds-G_K}
		\MoveEqLeft A\|c\|^2 \leq \langle \widetilde{G}_K(\w) c,c \rangle \leq B\|c\|^2,\\
		&\qquad\qquad\forall \,c\in \text{span}\{\mathcal{T}\phi_1(\w),\dots,\mathcal{T}\phi_k(\w)\},\\
		&\qquad\qquad\qquad\qquad \text{	for a.e. }\w\in [-1/2,1/2]^d.
	\end{align}
	Now, for ${c\in \text{span}\{\mathcal{T}\phi_1(\w),\dots,\mathcal{T}\phi_k(\w)\}}$,
	\begin{align*}
		\langle \widetilde{G}_K(\w) c,c \rangle &= \left\langle \hat{\Phi}(\w) \overline{Z_\Phi(X,\w)} Z_\Phi(X,\w)^t  \hat{\Phi}(\w)^* c,c\right\rangle \\
		& = \left\langle \overline{Z_\Phi(X,\w)} Z_\Phi(X,\w)^t \,\hat{\Phi}(\w)^*c,  \hat{\Phi}(\w)^* c\right\rangle.
	\end{align*}
	To conclude the proof, we observe that $\hat{\Phi}(\w)^*\left(\text{span}\{\mathcal{T}\phi_1(\w),\dots,\mathcal{T}\phi_k(\w)\}\right) = \C^k$. Indeed, since $E(\Phi)$ is an orthonormal basis, then $\{\mathcal{T}\phi_1(\w),\dots,\mathcal{T}\phi_k(\w)\}$ is an orthonormal system in $\ell^2(\Z^d)$ for a.e. $\w\in [-1/2,1/2]^d$.
\end{proof}

To prove the probability estimates for the sampling inequalities in \eqref{eq:sampling-ineq-m}, we will make use of Proposition \ref{propo:sampling-condition} and the matrix Bernstein inequality, which we quote from \cite{Tro12}. Here, and henceforth, for a matrix $X\in \C^{k\times k}$ we write
	$$\|X\|:=\sup\left\{\|Xa\|_2\,:\, a\in\C^k,\, \|a\|_2=1\right\}.$$

\begin{theo}\label{thm:bernstein}
	Let $\{X_1,\dots,X_m\}$ be a sequence of independent, random self-adjoint $k\times k$ matrices. Suppose that
	\begin{equation}
		\mathbb{E}X_r = 0 \quad\text{and} \quad \|X_r\|\leq B\quad \text{a.s.},\quad \forall \,r=1,\dots,m.
	\end{equation}
	and let
	\begin{equation}
		\sigma^2 = \Big\|\sum_{r=1}^m \mathbb E(X_r^2)\Big\|.
	\end{equation}
	Then for all $\nu\geq 0$
	\begin{equation}
		\mathbb P\bigg(\lambda_{\max}\bigg(\sum_{r=1}^{m}X_r\bigg)\geq \nu\bigg)\leq k \exp{\left(-\frac{\nu^2/2}{\sigma^2+B\nu/3}\right)}.
	\end{equation}
\end{theo}

From now on, we consider a set $X=\{x_1,\dots,x_m\}$ of independent random variables that are uniformly distributed in  $[-1/2,1/2]^d$. Associated with this set, we have the family of $k\times k$ random matrices
\begin{equation}\label{eq:T(w)}
	T(\w) := \overline{Z_\Phi(X,\w)} Z_\Phi(X,\w)^t,
\end{equation}
defined for a.e. $\w\in [-1/2,1/2]^d.$
Additionally, for  $r=1,\dots, m$ and a.e. $\w\in[-1/2,1/2]^d$, let $T_r(\w)\in\C^{k\times k}$ be the random matrix with entries
\begin{equation}\label{eq:T_r}
	(T_r(\w))_{i,j} :=  \overline{Z_{\phi_i}(x_r,\w)}Z_{\phi_j}(x_r,\w),\quad i,j\in\{1,\dots,k\}.
\end{equation} 
Then $T_r(\w)$ is positive semi-definite and $$T(\w) =\sum_{r=1}^m T_r(\w).$$
By \eqref{eq:Zak-Fourier} we can express $T_r(\w)$ in terms of the Zak transform of $\hat{\phi}_i$:
for every $i,j\in\{1,\dots,k\}$,
\begin{align*}
	(T_r(\w))_{i,j} =  \overline{Z_{\hat\phi_i}(\w,-x_r)}Z_{\hat\phi_j}(\w,-x_r),
\end{align*} 
for a.e. $\w\in [-1/2,1/2]^d$.
Finally, we introduce the matrices
$$X_r(\w) := T_r(\w) - I_k,\quad  r=1,\dots,m,$$
where $T_r(\w)$ is given by \eqref{eq:T_r} and $I_k$ denotes the identity matrix in $\C^{k\times k}$.

\begin{lemma}\label{lem:ineq-X_r} 
	For every $r=1,\dots,m$ and for a.e. ${\w\in [-1/2,1/2]^d}$ the following holds:
	\begin{enumerate}
		\item[(a)] $\mathbb E X_r(\w)=0,$
		\item[(b)] $\left\|X_r(\w)\right\| \leq kC^2 + 1$ \,a.s.,  and
		\item[(c)] $\sigma^2=\left\|\sum_{r=1}^m \mathbb E(X_r(\w)^2)\right\|\leq m  (kC^2-1).$
	\end{enumerate}	
\end{lemma}

\begin{proof}
	To prove (a) we show that $\mathbb{E}(T_r(\w))= I_k$
	for a.e. ${\w\in[-1/2,1/2]^d}$.
	To this end, let $i,j\in\{1,\dots, k\}$, and compute, for a.e. ${\w\in [-1/2,1/2]^d}$,
	\begin{align}
		\MoveEqLeft \mathbb{E} \bigg(T_r(\w)\bigg)_{i,j} 
		= \int_{[-\frac{1}{2},\frac{1}{2}]^d}  \overline{Z_{\hat\phi_i}(\w,-x)}Z_{\hat\phi_j}(\w,-x)\,dx\\
		&=\int_{[-\frac{1}{2},\frac{1}{2}]^d}  \bigg(\overline{\sum_{\ell\in\Z^d}\hat\phi_i(\w+\ell)\, e^{2\pi i \langle \ell,x\rangle}}\bigg) \\ 
		&\qquad\times\bigg(\sum_{\ell'\in\Z^d}\hat\phi_j(\w+\ell')\, e^{2\pi i \langle \ell',x\rangle}\bigg)  \,dx \\
		&=\sum_{\ell\in\Z^d}  \overline{\hat\phi_i(\w+\ell)}\,\hat\phi_j(\w+\ell) \\
		&=\big\langle\mathcal T \phi_j(\w),\mathcal T \phi_i(\w) \big\rangle,
	\end{align}
	where we used the notation from \eqref{eq:fiber}. Since $E(\Phi)$ is orthonormal, $\{\mathcal T \phi_1(\w),\dots, \mathcal T \phi_k(\w)\}$ is also orthonormal for a.e. $\w\in [-1/2,1/2]^d$, and the claim follows.
	
	Let us prove (b). By \eqref{eq:phi-to-psi} and Condition \ref{item:cond-2},
	\begin{align}
		\left\|T_r(\w)\right\|&\leq \bigg(\sum_{i=1}^k\sum_{j=1}^k \left| \overline{Z_{\hat\phi_i}(\w,-x_r)}Z_{\hat\phi_j}(\w,-x_r)\right|^2\bigg)^{1/2}\\
		&=\sum_{i=1}^k \left|Z_{\hat\psi_i}(\w,-x_r)\right|^2 
		\leq\sum_{i=1}^k \bigg(\sum_{\ell\in\Z^d} |\hat\psi_i(\w+\ell)|\bigg)^2 \\
		&\leq k C^2.\label{eq:bound-T_r}
	\end{align}
	Hence,
	\begin{align}
		\left\|X_r(\w)\right\| = \|T_r(\w) - I_k\|\leq \|T_r(\w)\| + 1
		&\leq k C^2 +  1.
	\end{align}
	To prove (c), we first note that since $T_r(\w)$ is a positive semi-definite matrix,
	$$T_r(\w)^2 \leq \|T_r(\w)\| T_r(\w).$$
	By \eqref{eq:bound-T_r}, $\|T_r(\w)\|\leq k C^2$ and therefore $T_r(\w)^2 \leq k C^2  T_r(\w)$ for a.e. $\w\in[-1/2,1/2]^d$. Since the expected value of a random positive semi-definite matrix is also positive semi-definite, 
	\begin{align}
		\mathbb{E} (T_r(\w)^2) \leq  k C^2 \,  \mathbb{E}(T_r(\w)) =  k C^2 I_k, 
	\end{align}
	for a.e. $\w\in[-1/2,1/2]^d$.
	Therefore,
	\begin{align}
		\mathbb E(X_r(\w)^2) = \mathbb E(T_r(\w)^2) - I_k
		\leq   k C^2  I_k - I_k,
	\end{align}
	and
	\begin{equation}
		\sigma^2=\bigg\|\sum_{r=1}^m \mathbb E(X_r(\w)^2)\bigg\|\leq m (k C^2 -1)
	\end{equation}
	for a.e. $\w\in[-1/2,1/2]^d$.
\end{proof}

We are ready to prove our main theorem.

\begin{proof}[Proof of Theorem \ref{thm:main-thm}]
	By Proposition \ref{propo:sampling-condition}, it is enough to show that, with probability at least $1-\eps$, \eqref{eq:sampling-ineq-Zak} hold for a.e. $\w\in [-1/2,1/2]^d$ with $A=m(1-\eps)$ and $B=m(1+\eps)$.
	
	\noindent {\bf Step 1}. \emph{(Concentration of measure)}. For each ${\w\in [-1/2,1/2]^d}$ where $T(\w)$ is defined, we apply Theorem \ref{thm:bernstein} with $\nu = m\alpha/2$ to obtain
	\begin{align}
		\MoveEqLeft\mathbb{P}\bigg( \min_{\|a\|_{\C^k}=1} \langle T(\w) a,a\rangle \leq m\left(1-\frac{\alpha}{2}\bigg)\right)\\ &\leq \mathbb{P} \bigg(  \lambda_{\min} \bigg( \sum_{r=1}^m X_r(\w) \bigg) \leq - \frac{m\alpha}{2} \bigg)
		\\
		&\leq k \exp\bigg(-\frac{m^2\alpha^2/8}{m (kC^2-1) + (kC^2+1)  m\alpha / 6}\bigg)\\ 
		&\leq k\exp\bigg(\frac{-\alpha^2m}{10kC^2}\bigg),
	\end{align} 
	and, similarly,
	\begin{align*}
		\MoveEqLeft\mathbb{P}\left( \max_{\|a\|_{\C^k}=1} \langle T(\w) a,a\rangle \geq m\left(1+\frac{\alpha}{2}\right)\right)\\ &\leq \mathbb{P} \left(  \lambda_{\max} \left( \sum_{r=1}^m X_r(\w) \right) \geq  \frac{m\alpha}{2} \right)\\
		&\leq k\exp\left(\frac{-\alpha^2m}{ 10 kC^2}\right).
	\end{align*}
	
	\noindent {\bf Step 2}. \emph{(Net bound)}. We now construct a net in $[-1/2,1/2]^d$ that avoids the zero measure set excluded by the essential suprema in Conditions \ref{item:cond-2} and \ref{item:cond-3}.
	
	More precisely, let $Z_1$ be the (zero measure) set of all ${\w\in[-1/2,1/2]^d}$ such that \[\sup_{i\in\{1,\dots,k\}} \sum_{\ell\in\Z^d} |\hat\psi_i(\w+\ell)| > C.\] In addition, let $E\subset [-1/2,1/2]^d\times [-1/2,1/2]^d$ be the (zero measure) set of all $(\w,\w')$ such that 
	$$
	\sup_{i\in\{1,\dots,k\}} \sup_{x\in[-1/2,1/2]^d} |Z_{\hat\psi_i}(\w,x)-Z_{\hat\psi_i}(\w',x)|> K\|\w-\w'\|_{\infty}.
	$$ 
	Let $Z_2\subset [-1/2,1/2]^d$ be the set of all $\w\in[-1/2,1/2]^d$ such that the measure of the $\w$-section 
	$$E_\w:=\{\w'\in [-1/2,1/2]^d \,:\, (\w,\w')\in E \}$$
	is not zero. Then, by Fubini, $Z_2$ has zero measure and so does $Z:=Z_1\cup Z_2$.
	
	Second, fix $\delta>0$, and $\eta>0$. We claim that there exists a net $\mathcal N_{\delta,\eta}\subset [-1/2,1/2]^d\setminus Z$ with the property that ${\#\mathcal N_{\delta,\eta} = \left\lceil\frac{1}{2\delta} \right\rceil^d}$ and for every $\w\in [-1/2,1/2]^d$ there exists $\w'\in \mathcal N_{\delta,\eta}$ such that $\|\w-\w'\|_\infty \leq \delta(1+\eta)$. Indeed, start with the set $A_\delta=2\delta \Z^d \cap [-1/2,1/2]^d$; if any point $a\in A_\delta$ belongs to $Z$, then we replace it with any other point taken from $[a-\delta\eta,a+\delta\eta]^d \cap [-1/2,1/2]^d \setminus Z$. The resulting set $\mathcal N_{\delta,\eta}$ satisfies the requirements.
	
	For $\w\in [-1/2,1/2]^d \setminus Z$, we select $\w'\in\mathcal N_{\delta,\eta}$ such that $\|\w-\w'\|_\infty\leq \delta(1+\eta)$ and estimate
	\begin{align}
		\MoveEqLeft\sup_{\|a\|_{\C^k} =1} \bigg| \langle (T(\w)-T(\w'))a,a \rangle \bigg| \\
		&\leq \bigg( \sum_{i,j=1}^{k} \bigg| \sum_{r=1}^m (T_r(\w))_{i,j} - (T_r(\w'))_{i,j} \bigg|^2 \bigg)^{\frac{1}{2}}\\
		&\leq \bigg( \sum_{i,j=1}^{k} \bigg( \sum_{r=1}^m \big|(T_r(\w))_{i,j} - (T_r(\w'))_{i,j}\big| \bigg)^2 \bigg)^{\frac{1}{2}}\\
		&=\bigg( \sum_{i,j=1}^{k} \bigg( \sum_{r=1}^m \big| \overline{Z_{\hat\psi_i}(\w,-x_r)}Z_{\hat\psi_j}(\w,-x_r)\\
		&\qquad\qquad\qquad-\overline{Z_{\hat\psi_i}(\w',-x_r)}Z_{\hat\psi_j}(\w',-x_r)\big| \bigg)^2 \bigg)^{\frac{1}{2}},\\ \label{eq:bound-diff-T(w)} 
	\end{align}
	where in the last equality we used \eqref{eq:phi-to-psi} (which, crucially, involves an exponential factor independent of $\w$).
	We now estimate the summands in \eqref{eq:bound-diff-T(w)}  using Conditions \ref{item:cond-2} and \ref{item:cond-3}:
	\begin{align}
		\MoveEqLeft\left| \overline{Z_{\hat\psi_i}(\w,-x_r)}Z_{\hat\psi_j}(\w,-x_r)-  \overline{Z_{\hat\psi_i}(\w',-x_r)}Z_{\hat\psi_j}(\w',-x_r)\right| \\
		&\leq |\overline{Z_{\hat\psi_i}(\w,-x_r)}| |Z_{\hat\psi_j}(\w,-x_r) - Z_{\hat\psi_j}(\w',-x_r)| \\
		&\quad + |Z_{\hat\psi_j}(\w',-x_r)| |\overline{Z_{\hat\psi_i}(\w,-x_r)} - \overline{Z_{\hat\psi_i}(\w',-x_r)}| \\
		&\,\leq 2C K \|\w-\w'\|_\infty.				
	\end{align}
	We combine this estimate with  \eqref{eq:bound-diff-T(w)} and we get
	\begin{equation}
		\sup_{\|a\|_{\C^k} =1}\left| \left\langle (T(\w)-T(\w'))a,a \right\rangle \right| \leq 2kmCK \delta(1+\eta).
	\end{equation}
	If $K=0$, then $T(\w) = T(\w')$ for a.e. $\w,\w'\in[-1/2,1/2]^d$, and we may skip the following steps until Equation \eqref{eq:bound-inequalities-a.e.w} simply replacing $\# \mathcal{N}_{\delta,\eta} =1$.  
	
	Suppose that $K\neq 0$, and let us set $$\delta=\min\left\{\frac{1}{2},\frac{\alpha}{4 kC K(1+\eta)}\right\}.$$  Recall that $\# \mathcal{N}_{\delta,\eta} = \big\lceil\frac{1}{2\delta} \big\rceil^d $ and thus
	\begin{equation}\label{eq:N_delta}
		\# \mathcal{N}_{\delta,\eta} \leq \bigg(\frac{2k C K(1+\eta)}{\alpha}  +1\bigg)^d.
	\end{equation}
	Moreover, for this choice of $\delta$, $$\sup_{\|a\|_{\C^k} =1}\big| \big\langle (T(\w)-T(\w'))a,a \big\rangle \big| \leq \frac{m\alpha}{2},$$ and therefore
	\begin{alignat}{2}
		 \min_{\w'\in \mathcal{N}_{\delta,\eta}} \min_{ \|a\|_{\C^k} = 1 }& \langle T(\w')a,a \rangle - \frac{m\alpha}{2} \\
		&\leq  \langle T(\w)a,a \rangle \\
		 &\leq \max_{\w'\in \mathcal{N}_{\delta,\eta}} \max_{ \|a\|_{\C^k} = 1 } \langle T(\w')a,a \rangle + \frac{m\alpha}{2}.
	\end{alignat}		
	Finally, we can compute the desired probability estimates. First,
	\begin{align}
		\MoveEqLeft\mathbb{P}\bigg( \underset{\w\in [-1/2,1/2]^d}{\text{\rm ess inf }} \min_{\|a\|_{\C^k}=1} \langle T(\w) a,a \rangle \leq m(1-\alpha) \bigg) \\
		&\leq \mathbb{P}\bigg(\min_{\w'\in\mathcal N_{\delta,\eta}}\min_{\|a\|_{\C^k}=1} \langle T(\w') a,a \rangle \leq m\bigg(1-\frac{\alpha}{2}\bigg) \bigg)\\
		&\leq \sum_{\w'\in\mathcal{N}_{\delta,\eta}} \mathbb{P}\bigg(\min_{\|a\|_{\C^k}=1} \langle T(\w') a,a \rangle \leq m\bigg(1-\frac{\alpha}{2}\bigg) \bigg)\\
		&\leq \# \mathcal{N}_{\delta,\eta}  \,k\exp\bigg(\frac{-\alpha^2m}{ 10 kC^2}\bigg).
	\end{align}
	In a similar manner,
	\begin{align}
		\MoveEqLeft\mathbb{P}\bigg( \underset{\w\in [-1/2,1/2]^d}{\text{\rm ess sup }} \max_{\|a\|_{\C^k}=1} \langle T(\w) a,a \rangle \geq m(1+\alpha) \bigg) \\
		&\leq \mathbb{P}\bigg(\max_{\w'\in\mathcal N_{\delta,\eta}}\max_{\|a\|_{\C^k}=1} \langle T(\w') a,a \rangle \geq m\left(1+\frac{\alpha}{2}\bigg) \right)\\
		&\leq \# \mathcal{N}_{\delta,\eta}  \,k\exp\bigg(\frac{-\alpha^2m}{ 10 kC^2}\bigg).
	\end{align}
	Therefore, recalling \eqref{eq:N_delta},
	\begin{align}\label{eq:bound-inequalities-a.e.w}
		\MoveEqLeft\mathbb{P}\bigg(\text{For a.e. } \w\in [-1/2,1/2]^d,\, \forall \, a\in\C^k,\, \|a\|_{\C^k} =1:\\
		& \,m(1-\alpha) \leq \langle T(\w) a,a \rangle \leq m(1+\alpha) \bigg) \\
		&\qquad\leq 1- 2\,\# \mathcal{N}_{\delta,\eta}  \, k\exp\left(\frac{-\alpha^2m}{ 10 kC^2}\right)
		\\
		&\qquad\leq 1- 2\,\bigg(\frac{2k C K(1+\eta)}{\alpha}  +1\bigg)^d  \, k\exp\left(\frac{-\alpha^2m}{ 10 kC^2}\right).
	\end{align}
	Letting $\eta \to 0^+$, we conclude that if
	\begin{equation}
		m\geq10 \cdot \frac{C^2}{\alpha^2}\cdot k\cdot  \log\left( \frac{2k}{\varepsilon} \left( \frac{2k C K}{\alpha}  +1\right)^d\right),
	\end{equation}
	then, with probability at least $1-\varepsilon$, the condition in \eqref{eq:sampling-ineq-Zak} is satisfied, and therefore the sampling inequalities \eqref{eq:sampling-ineq-m} hold.
\end{proof}

\section{Proofs of Corollary \ref{coro:union-cubes}, Theorem \ref{coro:k-tiles}, and Theorem \ref{teo general set}}\label{sec:thm3}

\begin{proof}[Proof of Corollary \ref{coro:union-cubes}]
	Let us check that the functions ${\psi_i:=\psi}$ defined by $\hat\psi = \chi_{[-1/2,1/2]^d}$ satisfy Conditions \ref{item:cond-1}-\ref{item:cond-3} in Theorem \ref{thm:main-thm}. First,
	$$\psi(x) = \prod_{n=1}^{d} \frac{\sin\left(\pi x(n)\right)}{\pi x(n)},
	\quad x=\left(x(1),\dots,x(d)\right),
	$$
	so, $\psi$ satisfies Condition \ref{item:cond-1}. Furthermore, for almost every ${\w\in[-1/2,1/2]^d}$,
	$$
	\sum_{\ell\in\Z^d} |\hat\psi(\w+\ell)| = \sum_{\ell\in\Z^d} \chi_{[-1/2,1/2]^d}(\w+\ell) = 1
	$$
	and
	$$
	Z_{\hat\psi}(\w,x) =1.
	$$
	This shows that Condition \ref{item:cond-2} holds with $C=1$, whereas Condition \ref{item:cond-3} holds with $K=0$. Set also $$\Phi=\{\phi_1,\dots,\phi_k\},\quad\phi_i = e^{2\pi i \langle \ell_i,\cdot \rangle} \psi,\quad i=1,\dots,k,$$ then $\|\phi_i\|_2 = \|\psi_i\|_2=1$ and $|\supp(\hat\phi_i) \cap \supp (\hat\phi_j)| =0$ for $i\neq j$. Hence $E(\Phi)=\{\phi_i(\cdot-\ell)\,:\, i=1,\dots,k,\, \ell\in\Z^d\}$ is an orthonormal system in $L^2(\R^d)$ while $V=S(\Phi)=PW_\Omega$, so we can invoke Theorem \ref{thm:main-thm}.
\end{proof}

	For the proof of Theorem  \ref{coro:k-tiles}, we will need the following lemma.
\begin{lemma}\label{lem:complex-index}
	Let $\Omega\subset\R^d$ be a compact set with tiling complexity index $N_\Omega$ and let $k = k_\Omega$ be given by \eqref{eq:k-level}. Then there exists a compact set $\Omega^*\supset \Omega$ that multi-tiles $\R^d$ at level $k$ with respect to the integer lattice, i.e.,
	$$\sum_{\ell\in\Z^d} \chi_{\Omega^*} (\w+\ell) = k,\quad \text{for a.e. }\w\in [-1/2,1/2)^d,$$ 
	such that
	$N_{\Omega^*} \leq N_{\Omega}$. 
\end{lemma}

\begin{proof}
	In \cite[Proposition 2.1]{BCHLMM18} it is shown that $\Omega$ can be enlarged to a compact set $\Omega^* \supset \Omega$ that multi-tiles $\R^d$ at level $k$ with respect to the integer lattice. Here we follow the argument of \cite{BCHLMM18} to show that the enlarged set can be chosen to also satisfy $N_{\Omega^*} \leq N_{\Omega}$. 
	
	Consider disjoint measurable sets $Q_1,\dots,Q_{N_\Omega}\subset [-1/2,1/2]^d$ and $B_1,\dots,B_{N_\Omega}\subset \Z^d$ such that the decomposition in \eqref{eq_compi} holds up to a null measure set. Since $k = k_\Omega$ (c.f. \eqref{eq:k-level}), we know that $\#B_n \leq k$ for every $n=1,\dots,N_\Omega$. We extend each $B_n$ to $\widetilde{B}_n\subseteq L_\Omega$ --- c.f.  \eqref{eq:L_Omega} --- by inserting $(k-\#B_n)$ distinct elements from $L_\Omega\setminus B_n$. Notice that, by definition, $L_\Omega$ has at least $k$ elements. Define $Q_0 = [-1/2,1/2]^d \setminus \bigcup_{n=1}^{N_\Omega} Q_n$, and set
	\begin{equation}\label{eq:Omega*}
		\Omega^* :=  \left((Q_1\cup Q_0) + \widetilde{B}_1\right) \cup \left(\bigcup_{n=2}^{N_\Omega} Q_n + \widetilde{B}_n\right).
	\end{equation}

	Observe that $\Omega \subseteq\Omega^*$, and $N_{\Omega^*} \leq N_{\Omega}$.  
	Finally, since $\#\widetilde{B}_n = k$ for every $n=1,\dots,N_{\Omega}$, we get that
	$\sum_{\ell\in\Z^d} \chi_{\Omega^*} (\w+\ell) = k,$ for a.e. $[-1/2,1/2)^d$.
\end{proof}

\begin{proof}[Proof of Theorem \ref{coro:k-tiles}]
	By Lemma \ref{lem:complex-index}, there exists a compact set $\Omega^* \supset \Omega$ that multi-tiles $\R^d$ at level $k$ with respect to the integer lattice such that  $N_{\Omega^*} \leq N_{\Omega}$.
	Thus, assume without loss of generality that $\Omega=\Omega^*$.
	
	Consider a decomposition \eqref{eq_compi} that achieves the tiling complexity index $N=N_\Omega$. Since $\Omega$ is a multi-tile of $\R^d$ at level $k$ with respect to the integer lattice,
	\begin{align}\label{eq_dex}	 [-1/2,1/2]^d = \bigcup_{n=1}^{N_\Omega} Q_n,
	\end{align}
	up to a null measure set, and $\# B_n = k$ for every ${n=1,\dots,N_\Omega}$. Let us write $B_n=\{b_1^n,\dots,b_k^n\}$ and set
	$$\Omega_i := \bigcup_{n=1}^{N_\Omega} Q_n + b_i^n,\quad i=1,\dots,k.$$
	Then $\Omega = \bigcup_{i=1}^k \Omega_i$ and $\left|\Omega_i \cap \Omega_{j}\right|=0$ if $i\neq j$. Set also $\hat\phi_i = \chi_{\Omega_i}$ and $\Phi=\{\phi_1,\dots,\phi_k\}$. Then $E(\Phi)$ is orthonormal and $PW_{\Omega}=S(\Phi)$. Moreover, Proposition \ref{propo:sampling-condition} is applicable because $PW_{\Omega}$ is a reproducing kernel Hilbert space (this result does not require the more stringent hypotheses of Theorem \ref{thm:main-thm}).
	
	Let $X=\{x_1,\dots,x_m\}$ be independent and uniformly distributed in  $[-1/2,1/2]^d$. Recall the definition of the random matrices $T(\w)$ and $T_r(\w)$ in \eqref{eq:T(w)} and \eqref{eq:T_r}. We now note that the matrices have in fact a very simple form, much in similarity to \cite{Mar06,Ko15,AAC15}. More precisely, for a.e. $\w\in Q_n$ with $n=1,\dots,k$, we have that for $i,j\in\{1,\dots,k\}$, and  $r=1,\dots,m$, 
	\begin{align}
		(T_r(\w))_{i,j} &=  \overline{Z_{\hat\phi_i}(\w,-x_r)}Z_{\hat\phi_j}(\w,-x_r)\\
		&= \bigg(\sum_{\ell\in \Z^d}  \overline{\hat{\phi}_i(\w+\ell)}e^{-2\pi i \langle x_r,\ell\rangle}\bigg)\\
		&\qquad\times\bigg( \sum_{\ell\in \Z^d}  \hat{\phi}_j(\w+\ell)e^{2\pi i \langle x_r,\ell\rangle}\bigg)\\
		&=  e^{-2\pi i \langle x_r,b_i^n\rangle} e^{2\pi i \langle x_r,b_j^n\rangle}\\
		&= e^{2\pi i \left\langle x_r,b_i^n-b_j^n\right\rangle}.\label{eq:T_r-k-tile}
	\end{align} 
	
	Inspecting \eqref{eq:T_r-k-tile} and \eqref{eq_dex}, we see that, discarding a null measure set of parameters $\w$, there exist only a finite number of configurations for the matrices $T(\w)= \sum_{r=1}^m T_r(\w)$, say
	$$T(\w_1),\dots,T(\w_{N_\Omega}),$$
	where $\w_n\in Q_n$. 
	Hence, we only need to check \eqref{eq:sampling-ineq-Zak} for those matrices.
	Proceeding as in the proof of Theorem \ref{thm:main-thm}, we see that 
	\begin{align}
		\MoveEqLeft\mathbb{P}\bigg( \text{For a.e. } \w\in [-1/2,1/2]^d,\, \forall \, a\in\C^k,\, \|a\|_{\C^k} =1: \\
		& m(1-\alpha) \leq \langle T(\w) a,a \rangle \leq m(1+\alpha) \bigg) \\
		&\qquad\leq 1- 2\,N_\Omega  \, k\exp\left(\frac{-\alpha^2m}{ 10 k}\right)
	\end{align}
	and the conclusion follows.
\end{proof}

\begin{proof}[Proof of Theorem \ref{teo general set}] Given $\rho>0$, let $\Omega_\rho$ denote the smallest set (with respect to the inclusion) that contains $\Omega$, and it has the form
	$$
	\bigcup_{i=1}^k [-\rho/2,\rho/2]^d+\ell_i,
	$$
	where $\ell_1,\dots,\ell_k$  are distinct points of the lattice $\rho\Z^d$. By a covering argument based on \cite[Theorems 5 and 6]{Ca05} --- see \cite[Lemma 2.1]{MRS23} --- we get the estimate
	$$
	|\Omega_\rho\setminus \Omega|\leq 2\ \frac{\beta_d (1+\sqrt{d-1})^2 11^d}{\kappa_{\partial \Omega}}\cdot\cH^{d-1}(\partial \Omega)\cdot \rho,
	$$
	where $\beta_d$ denotes the volume of the $d$-dimensional unit ball.  Let $C_d:= 2\beta_d (1+\sqrt{d-1})^2 11^d$ and let
	\begin{equation}\label{el rho}
		\rho=\frac{\kappa_{\partial \Omega}}{C_d} \cdot \frac{|\Omega|}{ \cH^{d-1}(\partial \Omega)},
	\end{equation}
	$\varepsilon>0$ and $0<\alpha<1$. Combining Corollary \ref{coro:union-cubes} with a dilation argument, if
	\begin{equation}\label{emes dominio general}
		m\geq \frac{10}{\alpha^2}\cdot k\cdot \log\left( \frac{2k}{\varepsilon}\right),
	\end{equation}
	and $x_1,\dots,x_m$ are independent random variables uniformly distributed in $[-(2\rho)^{-1},(2\rho)^{-1}]^d$, 
	then the following inequalities
	\begin{equation}\label{eq:sampling-ineq-general rho}
		m(1-\alpha)\,\rho^d \,\|f\|^2 \leq  \sum_{r=1}^m  \sum_{\ell\in\frac{1}{\rho}\Z^d} |f(x_r+\ell)|^2 \leq m(1+\alpha)\,\rho^d\, \|f\|^2
	\end{equation}
	hold for every $f\in PW_{\Omega_\rho}$ with probability at least $1-\varepsilon$. In particular, they also hold for every $f\in PW_\Omega$. Since $k\leq \frac{2|\Omega|}{\rho^d}$, if
	$$
	m\geq \frac{20}{\alpha^2} \cdot \frac{|\Omega|}{\rho^d} \cdot \log\left(\frac{4}{\varepsilon} \cdot \frac{|\Omega|}{\rho^d}\right),
	$$ 
	then $m$ also satisfies \eqref{emes dominio general}, and therefore the sampling inequalities \eqref{eq:sampling-ineq-general rho} hold for such $m$. 
\end{proof}

\section{Conclusion}
We considered a multi-dimensional version of the random interleaved sampling strategy and showed that it is effective for very general signal models generated by a few basic frequency profiles. As shown by Landau's foundational work \cite{landau1967sampling,La67} and subsequent extensions \cite{MR2224392,MR3742438}, the number of such frequency profiles determines the so-called \emph{Nyquist rate}, an information theoretic limit below which no sampling strategy can succeed. Our results support the conventional wisdom in sampling theory that a generic acquisition strategy with a moderate amount of oversampling leads to numerically stable and reliable reconstructions, while sampling patterns exactly achieving Landau's benchmark may be more fragile and less practically attractive. We formally confirmed the validity of that intuition for signals with possibly complex and disconnected spectra.

\section*{Acknowledgements}

{J.A. was supported by grants: PICT 2019 0460 (ANPCyT), PIP112202101 00954CO (CONICET), 11X829 (UNLP), PID2020-113048GB-I00 (MCI), and María Zambrano CA3/RSUE/2021-00185 (Ministerio de Universidades, el Plan de Recuperación, Transformación y Resiliencia y la UAM).
	D.C. was supported by the European Union’s programme Horizon Europe, HORIZON-MSCA-2021-PF-01, Grant agreement No. 101064206. J.L.R. gratefully acknowledges support from the Austrian Science Fund (FWF): Y 1199.

\bibliographystyle{IEEEtran}  
\bibliography{ransis-arxiv.bib}  

\vfill
	
	Jorge Antezana is with Centro de Matemática de La Plata, Facultad de Ciencias Exactas, Universidad Nacional de La Plata, 1900 La Plata, Bs.As., Argentina, Instituto Argentino de Matemática  “Alberto P. Calderón", (CONICET), 1083 CABA, Argentina, and also with Departamento de matemática, Facultad de Ciencias, Universidad Autónoma de Madrid, 28049 Madrid, Spain (e-mail: jorge.antezana@uam.es).
	
	Diana Carbajal is with the Faculty of Mathematics, University of Vienna,
	1090 Vienna, Austria (e-mail: diana.agustina.carbajal@univie.ac.at).

	Jos\'e Luis Romero is with the Faculty of Mathematics, University of Vienna, 
	1090 Vienna, Austria, and with the Acoustics Research Institute, Austrian Academy of
	Sciences, Vienna, 1040, Austria (e-mail: jose.luis.romero@univie.ac.at).
	}

\end{document}